\documentclass{article}

\usepackage[centertags]{amsmath}
\usepackage{amsfonts}
\usepackage{amssymb}
\usepackage{amsthm}
\usepackage{newlfont}
\usepackage{mathtools}
\usepackage[top=1in, bottom=1in, left=1in, right=1in]{geometry}
\usepackage{tikz}
\usepackage{scalerel}

\theoremstyle{plain}
\newtheorem{theorem}{Theorem}[section]

\newtheorem{corollary}[theorem]{Corollary}
\newtheorem{lemma}[theorem]{Lemma}

\theoremstyle{definition}
\newtheorem{definition}[theorem]{Definition}

\newtheorem{example}[theorem]{Example}
\newtheorem{remark}[theorem]{Remark}

\newcommand{\R}{\mathbb{R}}

\newcommand{\Z}{\mathbb{Z}}
\newcommand{\C}{\mathbb{C}}
\newcommand{\PPP}{\mathbb{P}}

\newcommand{\diff}{\setminus}
\newcommand{\Hidden}[1]{}

\newcommand{\PP}{\mathcal{P}}

\newcommand{\M}{\tilde{M}}

\DeclareMathOperator{\Tr}{Tr}

\begin{document}

\tikzstyle{every node}=[circle,fill=black,inner sep=1pt]

\title{Pretty good quantum state transfer in asymmetric graphs via potential}
\author{Or Eisenberg\footnote{Department of Mathematics, Harvard University, Cambridge MA, oreisenberg@college.harvard.edu.  Supported in part by the Herchel Smith Fellowship.}~~~~~Mark Kempton\footnote{Center of Mathematical Sciences and Applications, Harvard University, Cambridge MA, mkempton@cmsa.fas.harvard.edu.  Supported in part by NSF ATD award DMS-1737873.}~~~~~Gabor Lippner\footnote{Department of Mathematics, Northeastern University, Boston MA, g.lippner@neu.edu}}
\date{}

\maketitle

\begin{abstract}
We construct infinite families of graphs in which pretty good state transfer can be induced by adding a potential to the nodes of the graph (i.e. adding a number to a diagonal entry of the adjacency matrix).  Indeed, we show that given any graph with a pair of cospectral nodes, a simple modification of the graph, along with a suitable potential, yields pretty good state transfer (i.e. asymptotically perfect state transfer) between the nodes.  This generalizes previous work, concerning graphs with an involution, to asymmetric graphs. 
\end{abstract}

\section{Introduction}

Transfer of quantum information with high fidelity through networks of locally coupled spin particles is an important problem in quantum information processing. Information can be considered as excitation in the network initiated at an input node, which then spreads according to the action of a Hamiltonian. The quality of the transfer depends on how strongly the excitation can then be concentrated at a given target node. The transfer is \emph{perfect} if there is a time $t$ at which the probability of the excitation being at the target node is 1. Initiated by Bose~\cite{Bose2003}, perfect state transfer has been extensively studied for various networks, both from the physical~\cite{kay2010} and the mathematical~\cite{godsil} point of view. It turns out that perfect state transfer is notoriously difficult to achieve. All known constructions involve very special networks and/or very special, highly non-uniform coupling strengths. In particular, it has been shown in~\cite{us} that for uniformly coupled chains of length at least four there can never be perfect state transfer between endpoints, not even in the presence of magnetic fields.

There is a somewhat less restrictive notion of a \emph{pretty good state transfer}, also referred to as ``almost perfect state transfer''. This requires the transfer probability to get arbitrarily close to 1 as time passes. While practically just as good as perfect state transfer, it is somewhat easier to achieve. The first examples of spin chains admitting pretty good state transfer appeared in~\cite{Vinet2012}. However, as demonstrated in~\cite{Godsil2012,Coutinho:PGST,Coutinho2016,vanBommel2016}, even pretty good state transfer is relatively rare in unmodulated spin chains with uniform couplings. 

In this paper we study pretty good state transfer in the single-excitation subspace of a spin network with $XX$ couplings, in the presence of a magnetic field. We will use graph theoretic terminology throughout the paper. We denote the network by $G$, the set of nodes (vertices) by $V(G)$ and the set of links (edges) by $E(G)$. The evolution of such a system is given by its Hamiltonian
\[ H_{XX} = \frac{1}{2}\sum_{ (i,j) \in E(G)} J_{ij} (X_i X_j + Y_i Y_j) + \sum_{i \in V(G)} Q_i \cdot Z_i ,\] where $X_i, Y_i, Z_i$ are the standard Pauli matrices, $J_{ij}$ denotes the strength of the $XX$ coupling between node $i$ and $j$, and the $Q_i$'s give the strength of the magnetic field yielding an energy potential at each node. 

It has been shown \cite{Bose2003,Christandl2004} that the restriction of this Hamiltonian to the single-excitation subspace is modeled by a continuous-time quantum walk on a graph with transition matrix, $U(t)$, given by
\[ U(t) = \exp(itA)
\]
where $A$ is the (possibly weighted) adjacency matrix of the graph. 

\begin{definition}
Let $G$ be a graph with vertices $u$ and $v$.
\begin{enumerate}
\item We say that $G$ admits \emph{perfect state transfer} (PST) from $u$ to $v$ if there is some time $t>0$ such that
\[
|U(t)_{u,v}| = 1.
\]
\item We say  $G$ admits \emph{pretty good state transfer} (PGST) from $u$ to $v$ if, for any $\epsilon>0$, there is a time $t>0$ such that
\[
|U(t)_{u,v}| > 1-\epsilon.
\]
\end{enumerate}
\end{definition} 

Our primary focus in this paper will be the effect of adding a potential induced by a magnetic field (the $Q_i$ above).  Previous work in \cite{invol} showed that in graphs with an involutional symmetry, one can often induce pretty good state transfer between a pair of nodes by appropriately choosing a potential on the vertex set. In \cite{srg}, it is shown that a potential can induce pretty good state transfer in strongly regular graphs as well.  The contribution of this paper is to show how to construct asymmetric, non-regular graphs that admit PGST between a pair of nodes $u,v$ if a suitably chosen potential is added to the adjacency matrix at $u$ and $v$.  The novelty of our constructions is that we do not require any symmetry or regularity in the graph.  In addition, our results apply to arbitrary real symmetric matrices (not just adjacency matrices).  We note that PST has been exhibited in asymmetric simple unweighted graphs in \cite{Tamon2011}.

A necessary condition for both PST and PGST between vertices $u$ and $v$ of a graph is that $u$ and $v$ must be \emph{cospectral} (see \cite{godsil,Coutinho:PGST}), that is $G\diff u$ and $G\diff v$ have the same spectrum.  One motivation for our previous work is that symmetry in a graph always naturally leads to cospectral vertices.  It also holds that pairs of vertices in strongly regular graphs are cospectral.  However, cospectral pairs can arise without any symmetry or regularity conditions.  Two relatively small examples are shown in Figure \ref{fig:cosp}.
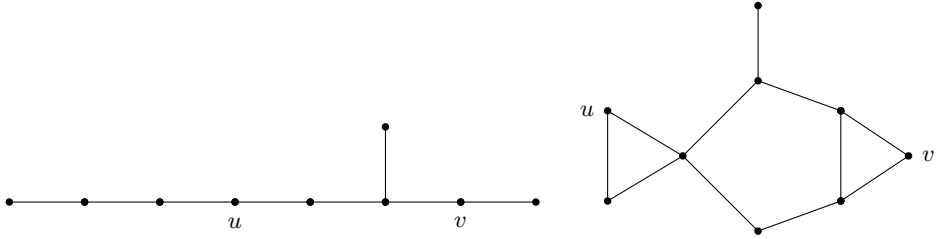
\begin{figure}
\begin{center}
\begin{tikzpicture}
\draw \foreach \x in {0,1,2,3,4,5,6}
{
(\x,0)node{}--(\x+1,0)node{}
}
(5,0)--(5,1)node{}
(3,-.1)node[fill=white, below]{\small $u$}
(6,-.1)node[fill=white, below]{\small $v$};
\end{tikzpicture}~~~
\begin{tikzpicture}
\draw (0,0)node{}--(-1,.6)node{}--(-1,-.6)node{}--(0,0)--(1,1)node{}--(2.1,.6)node{}--(2.1,-.6)node{}--(1,-1)node{}--(0,0)
(1,1)--(1,2)node{}
(2.1,.6)node{}--(3,0)node{}--(2.1,-.6)node{}
(-1.1,.6)node[fill=white, left]{\small $u$}
(3.1,0)node[fill=white, right]{\small $v$};
\end{tikzpicture}
\end{center}
\caption{\label{fig:cosp} Asymmetric, irregular graphs with cospectral nodes.  In each graph, the cospectral vertices are labeled $u$ and $v$.}
\end{figure}

Our constructions come in two types, based on the following two observations concerning cospectral vertices.  First, in a graph with an equitable partition (defined in Section \ref{sec:eq}) with a part containing exactly two vertices, those two vertices are cospectral (see Figure \ref{fig:eq} for an example).  Second, given two graphs with a cospectral pair, the vertices remain cospectral in the graph obtained by ``gluing" the two graphs together along those vertices (see Lemma \ref{lem:glue} below).  We are able to show that given any graph with a pair of cospectral vertices, a simple modification of the graph, together with an appropriately chosen potential on the vertex set, yields PGST between those vertices.  See Corollary \ref{cor:equit}, and Theorems \ref{thm:gluing}, \ref{thm:glue_pot}, and \ref{thm:change_trace} below for the precise details.

\begin{figure}
\begin{center}
\begin{tikzpicture}
\draw \foreach \x in {22.5,67.5,...,337.5}
{
(\x:2)node{}--(\x+45:2)node{}
}
\foreach \x in {22.5,67.5,112.5,247.5}{
(.5,.5)node{}--(\x:2)
}
\foreach \x in {157.5,202.5,292.5,337.5}{
(-.5,-.5)node{}--(\x:2)
}
(.6,.5)node[fill=white, below]{\small $u$}
(-.4,-.5)node[fill=white, above]{\small $v$};
\end{tikzpicture}
\end{center}
\caption{\label{fig:eq} A graph with an equitable partition with a part of size 2 (vertices $u$ and $v$), and hence a cospectral pair (but no involution swapping $u$ and $v$).}
\end{figure}
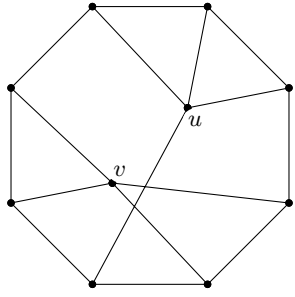

The key tool in our analysis is Theorem \ref{thm:tr/deg} below, which takes advantage of the fact that the characteristic polynomial of the adjacency matrix for a graph with cospectral nodes has a factorization.  We give a simple, efficiently computable condition on the factors that implies PGST. Note, however, that the converse of Theorem \ref{thm:tr/deg} does not hold in general.  Another critical piece in our proofs, of independent interest, is Lemma \ref{lem:strongly_cospectral_perturbation}, which shows that adding a transcendental potential to a pair of cospectral nodes actually makes them \emph{strongly} cospectral (see Section \ref{sec:prelim} for the definition).

\section{Preliminaries}\label{sec:prelim}

Let $M$ be a symmetric matrix with entries in a field $\mathcal{F}$. We use $\phi_M \in \mathcal{F}[t]$ to denote its characteristic polynomial. The rows and columns of $M$ will typically be indexed by the nodes of a finite graph. We will use $V(M)$ to denote the set of row/column indices of $M$, so we can think of $M \in \R^{V(M)\times V(M)}$. If $S \subset V(M)$, we write $M_S$ for the symmetric submatrix obtained from $M$ by \emph{removing} the rows and columns indexed by $S$.

\begin{definition}\label{def:wmz} For any vector $z$, let $W(M,z) = \langle z, Mz, M^2 z,\dots \rangle$ denote the $M$-invariant subspace generated by $z$. Let us denote by $\rho_z = \rho_{z,M} \in \mathcal{F}[t]$ the minimal polynomial of $M$ relative to $z$, that is, the smallest degree polynomial such that $\rho_z(M)z = 0$. It is well-known that $\rho_z$ divides the usual minimal polynomial of $M$ and that the degree of $\rho_z$ equals the dimension of $W(M,z)$. 
\end{definition}

\subsection{Cospectral nodes}

\begin{definition} Let $M$ be a symmetric matrix. Two indices $u, v \in V(M)$ are \emph{cospectral} if $\phi(M_u) = \phi(M_v)$.
\end{definition}

\begin{lemma}[Theorem 3.1 of \cite{godsil_smith_2017}]
 Let $M$ be a symmetric matrix, and let $u,v \in V(M)$. Let $M = \sum_{\lambda}\lambda E_{\lambda}$ be the spectral decomposition of $M$. Here $E_\lambda$ denotes the projections onto the eigenspaces of $M$ corresponding to the eigenvalue $\lambda$. We denote the characteristic vectors of $u$ and $v$ by $e_u,e_v$  respectively. 
 The following are equivalent:
\begin{enumerate}
\item  $u$ and $v$ are cospectral.
\item $(E_\lambda)_{u,u} = (E_\lambda)_{v,v}$ for all $\lambda$. 
\item $M^k(u,u) = M^k(v,v)$ for all $k$.
\item $W(M,e_u+e_v)$ is orthogonal to $W(M,e_u - e_v)$. 
\end{enumerate}
\end{lemma}

\begin{definition}\label{def:ppm}
We define $P_+$ to be the minimal polynomial of $M$ relative to $e_u + e_v$, and $P_-$ to be the minimal polynomial of $M$ relative to $e_u - e_v$. 
\end{definition}

\begin{lemma}\label{lem:cospectral_decomposition}
Given a symmetric matrix $M$ and cospectral indices $u,v \in V(M)$, the characteristic polynomial of $M$ decomposes as
\[ \phi_M = P_+ \cdot P_- \cdot P_0, \]
where $P_+$ and $P_-$ have no multiple roots, and there is an orthonormal basis of eigenvectors of $M$ such that:
\begin{enumerate}
\item for each root $\lambda$ of $P_+$ the basis contains a unique eigenvector $\varphi$ with eigenvalue $\lambda$ and $\varphi(u) = \varphi(v) \neq 0$,
\item for each root $\lambda$ of $P_-$ the basis contains a unique eigenvector $\varphi$ with eigenvalue $\lambda$ and $\varphi(u) = - \varphi(v) \neq 0$,
\item for each root $\lambda$ of $P_0$ with multiplicity $k$ the basis contains exactly $k$ eigenvectors with eigenvalue $\lambda$ all of which vanish on both $u$ and $v$.
\end{enumerate}
In particular the degree of $P_{\pm}$ is the same as the dimension of the space $W(M,e_u \pm e_v)$. 
\end{lemma} 

\begin{proof}
Since $M$ is diagonalizable, its minimal polynomial doesn't have multiple roots, and hence neither does $P_+$ nor $P_-$. The roots of $P_+$ are exactly those eigenvalues $\lambda$ for which $E_\lambda(e_u+e_v) \neq 0$, and for such $\lambda$ the eigenvector $\varphi = E_\lambda(e_u+e_v)$ satisfies that $\varphi(u) = \varphi(v)$. Similarly for $P_-$. By cospectrality of $u$ and $v$, the eigenvectors obtained for $P_+$ and for $P_-$ are pairwise orthogonal. Finally, extending to an orthogonal basis for $M$, it is clear that each remaining eigenvector satisfies $\varphi(u)= \varphi(v) = 0$.
\end{proof}

\begin{remark}
Since the coefficients of $P_+$ give the unique linear dependency among $e_u + e_v, M(e_u+e_v), \dots, M^k(e_u+e_v)$, they belong to the same field as the entries of $M$.  The same is true for $P_-$, and thus for $P_0$.
\end{remark}

\begin{definition}
The indices $u,v \in V(M)$ are \emph{strongly cospectral} if $\varphi(u) = \pm \varphi(v)$ for every eigenvector $\varphi$ of $M$.
\end{definition}

\begin{lemma} The following are equivalent:
\begin{enumerate}
\item $u$ and $v$ are strongly cospectral.
\item $u$ and $v$ are cospectral, and $P_+$ and $P_-$ do not have any common roots.
\item $E_\lambda e_u = \pm E_\lambda e_v$ for all $\lambda$.
\end{enumerate}
\end{lemma}

\subsection{Pretty good state transfer}

The discrete Schr\"odinger equation, for an $n\times n$ matrix $M$, is given by
\[ \partial_t \psi_t = i M \psi_t,\]
where $\psi_t \in \C^n$. The solution of this equation can be written in the form 
\[ \psi_t = e^{itM}\psi_0.\]

\begin{definition}
$M$ has PGST from $u$ to $v$ if $\psi_0 = e_u$ implies that $\limsup_{t\to \infty} |\psi_t(v)| = 1$, or equivalently if $\limsup_{t\to \infty} | e^{itM}(u,v)| =1$.
\end{definition}

The following is a characterization of PGST (see Theorem 2 in~\cite{Coutinho:PGST}).

\begin{lemma}\label{lem:eig}
Let $u,v \in V(M)$ for the symmetric matrix $M$.  Then pretty good state transfer from $u$ to $v$ occurs if and only if the following two conditions are satisfied:
\begin{enumerate}
\item The indices $u$ and $v$ are strongly cospectral.
\item Let $\{\lambda_i\}$ be the roots of $P_+$, and $\{\mu_j\}$ the roots of $P_-$. Then for any choice of integers $\ell_i$, $m_j$ such that 
\begin{align*}
\sum_i \ell_i\lambda_i +\sum_j m_j\mu_j = 0\\
\sum_i\ell_i +\sum_j m_j =0,
\end{align*}
we have
\[
\sum_i m_i \text{ is even.}
\]
\end{enumerate}
\end{lemma}

Note that the first conditions could be weakened to just cospectral, since the second condition implies that $P_+$ and $P_-$ do not share any roots, so this implies strongly cospectral if the nodes are cospectral. 


The following theorem generalizes a result from \cite{invol} and a lemma from \cite{srg}.
\begin{theorem}\label{thm:tr/deg}
Let $M$ be a symmetric matrix with strongly cospectral indices $u, v \in V(M)$, and assume that $P_+$ and $P_-$ are irreducible polynomials.  Then if 
\[
\frac{\Tr(P_+)}{deg(P_+)}\neq\frac{\Tr(P_-)}{deg(P_-)},
\]
where $\Tr$ denotes the trace (i.e. the sum of roots) of a polynomial, then there is PGST from $u$ to $v$.
\end{theorem}
\begin{proof}
Our proof uses a technique from \cite{invol}.

Suppose we have integers $\ell_i$, $m_j$ satisfying \begin{align*}
\sum_i \ell_i\lambda_i +\sum_j m_j\mu_j = 0\\
\sum_i\ell_i +\sum_j m_j =0.
\end{align*}
To use Lemma \ref{lem:eig}, we wish to show 
$\sum_i \ell_i $
is even.

We will use a tool from Galois theory called the \emph{field trace} of a field extension.  For a Galois field extension $K$ of $F$, we define $\Tr_{K/F}: K \rightarrow F$ by 
\[
\Tr_{K/F}(\alpha) = \sum_{g\in Gal(K/F)}g(\alpha).
\]  The field trace is the trace of the linear map taking $x\mapsto \alpha x$. In Lemma \ref{lem:field_trace} of the Appendix we record a few basic facts about the field trace that we will use.

Now, let $F$ be the base field (the field containing all the entries of $M$), let $L/F$ be the splitting field for $P_+$, $J/F$ the splitting field for $P_-$, and $K/F$ the smallest field extension containing both $L$ and $J$.  Let us denote $r = deg(P_+)$ and $s= deg(P_-)$.  Since $P_+$ and $P_-$ are irreducible, then $L$ and $J$ are Galois extensions of $F$. Let us examine the field trace of the individual roots of $P_+$ and $P_-$.  We have 
\[
\Tr_{L/F}(\lambda_i) = \sum_{g\in Gal(K/F)}g(\lambda_i)
\]
and since $L$ is a Galois extension, the group acts transitively on the roots of $P_+$, so each of the $\lambda_k$'s shows in this sum, and each will appear $|Gal(L/F)|/r$ times.  Thus
\[
\Tr_{L/F}(\lambda_i) = \frac{[L:F]}{r}\sum_k\lambda_k.
\]
Note further that $\sum_k\lambda_k=\Tr(P_+)$, so we have shown
\[
\Tr_{L/F}(\lambda_i) = \frac{[L:F]}{r}(\Tr(P_+))
\]
for any $i$. In a similar way, by examining $P_-$ we obtain
\[
\Tr_{J/F}(\mu_j) = \frac{[J:F]}{s}(\Tr(P_-))
\]
for any $j$.  

Now apply the field trace to our linear combination of the $\lambda_i$ and $\mu_j$, and using the properties above we have,
\begin{align*}
0&=\Tr_{K/F}\left(\sum\ell_i\lambda_i+\sum m_j\mu_j\right)\\ &= \Tr_{K/F}\left(\sum\ell_i\lambda_i\right)+\Tr_{K/F}\left(\sum m_j\mu_j\right)\\
&= [K:L]\Tr_{L/F}\left(\sum\ell_i\lambda_i\right)+[K:J]\Tr_{J/F}\left(\sum m_j\mu_j\right)\\
&=[K:L]\sum \ell_i\Tr_{L/F}(\lambda_i) + [K:J]\sum m_j\Tr_{J/F}(\mu_j)\\
&=\frac{[K:L][L:F]}{r}\Tr(P_+)\sum\ell_i + \frac{[K:J][J:F]}{s}\Tr(P_-)\sum m_j\\
&=[K:F]\left(\frac{\Tr(P_+)}{r}\sum\ell_i + \frac{\Tr(P_-)}{s}\sum m_j\right)
\end{align*}

This, along with our assumption at the beginning give us
\begin{align*}
\frac{\Tr(P_+)}{r}\sum\ell_i + \frac{\Tr(P_-)}{s}\sum m_j&=0\\
\sum\ell_i + \sum m_j&=0.
\end{align*}
This is a system of two equations in the variable $\sum\lambda_i$, $\sum \mu_j$, and so if
\[
\frac{\Tr(P_+)}{r}\neq\frac{\Tr(P_-)}{s},
\]
then these two equations are linearly independent, and we obtain
\[\sum \lambda_i=\sum m_j=0.\]   In particular, each sum is even, so Lemma \ref{lem:eig} implies that we get pretty good state transfer.  
\end{proof}

\section{Diagonal perturbation}

In this section we investigate how, given a symmetric matrix $M$ with cospectral indices $u,v \in V(M)$,  adding a diagonal matrix $D$ to $M$ can be used to achieve \emph{strong} cospectrality of $u,v$ and \emph{irreducibility} of $P_+$ and $P_-$. We are going to do this by choosing $D$ to have two non-zero values only. To establish notation, for any set of indices $S \subset V(M)$, let $D_S$ denote the diagonal matrix with 1s in the positions belonging to $S$ and 0s elsewhere.  Let $D = Q \cdot D_{uv}$.

\begin{lemma}\label{lem:cospectral_perturbation} If $u,v \in  V(M)$ are cospectral indices for $M$, then they are also cospectral for $M+D$.
\end{lemma}

\begin{proof}
\[ \phi_{(M+D)_u} = \phi_{M_u} - Q\cdot \phi_{M_{uv}} = \phi_{M_v} - Q \cdot \phi_{M_{uv}} = \phi_{(M+D)_v}.\]
\end{proof}

\subsection{Achieving strong cospectrality}

The benefit of adding such a diagonal perturbation is that we can actually turn a pair of cospectral indices into strongly cospectral ones.

\begin{lemma}\label{lem:strongly_cospectral_perturbation} Let $M$ be a symmetric matrix with connected support whose entries are in a field $\mathcal{F}\leq \R$, and assume $u,v \in V(M)$ are cospectral. Suppose $Q$ is transcendental over $\mathcal{F}$, and $D = Q \cdot D_{uv}$, then $u$ and $v$ are strongly cospectral for $M+D$.
\end{lemma}

\begin{proof}
By simple expansion, and using that cospectrality means $\phi_{M_u} = \phi_{M_v}$, we can write
\[ \phi_{M+D} = \phi_M + 2Q \phi_{M_v} + Q^2 \phi_{M_{uv}}.\]
 By~\cite[Lemma 8.4]{godsil_smith_2017} it is sufficient to show that $\phi_{M_{uv}} / \phi_{M+D}$ has only simple poles. Proving this by contradiction we can assume that $\phi_{M+D}$ is not irreducible. However, for transcendental $Q$, it then has to factor over $\mathcal{F}[Q]$. So we can either write
\[ \phi_{M+D} = h \cdot (Q f_1 + f_0) \cdot (Q g_1 + g_0),\] or possibly \[ \phi_{M+D} = h \cdot (Q^2 f_2 + Q f_1 +f_0), \] where all the factors but $h$ are irreducible and $f_i, g_j, h \in \mathcal{F}[x]$.
In both cases it follows that $\phi_{M_{uv}} = h \cdot \psi$ and thus \[ \frac{\phi_{M_{uv}}}{ \phi_{M+D} } = \frac{\psi }{ (Q f_1 + f_0)\cdot (Q g_1 + g_0)}\]  or \[\frac{\phi_{M_{uv}}}{ \phi_{M+D}} = \frac{\psi}{ (Q^2 f_2 + Q f_1 + f_0)}.\] The second option immediately implies that all poles are simple. In order to have non-simple poles in the first case, the two irreducible factors in the denominator must have a common root, but then they must be identical. This means that $\phi_M = h f_0^2$ , $\phi_{M_u} = \phi_{M_v} = h f_1 f_0$, $\phi_{M_{u,v}} = h f_1^2$. Thus  $\psi_{u,v} = \phi_{M_u} \phi_{M_v} - \phi_M \phi_{M_{u,v}} = 0$ so, by \cite[Lemma 1.1, Chapter 4.1]{GodsilAlgebraicCombinatorics}, we get $(E_\lambda)_{u,v} = 0$ for all $\lambda$, so $(M^k)(u,v) = 0$ for all $k$. This contradicts the assumption that $M$ has connected support, so $u,v$ must be strongly cospectral for $M$.
\end{proof}

\subsection{Trace}

We see from Theorem~\ref{thm:tr/deg} that the trace of $P_+$ and $P_-$ can be useful in proving that there is PGST between two cospectral nodes. In this section we prove some important properties of $\Tr P_\pm$ under diagonal perturbation. 

\begin{lemma}\label{lem:traceppm} Let $M$ be a symmetric matrix whose elements are in a field $\mathcal{F} \leq \R$. Suppose $u, v \in V(M)$ are cospectral for $M$ and let $D = Q \cdot D_{uv}$ where $Q\in \R$ is transcendental over $\mathcal{F}$. Let $\phi_{M+D} = P_+ \cdot P_- \cdot P_0$ as in Lemma~\ref{lem:cospectral_decomposition}. Then $\Tr P_+ - Q \in \mathcal{F}$ and $\Tr P_- - Q \in \mathcal{F}$. 
\end{lemma}

\begin{proof}
Let us recall that $P_\pm$ are the minimal polynomials of $M+D$ relative to $e_u \pm e_v$, and that $\phi_{M+D} = P_+ \cdot P_- \cdot P_0$. Since $Q$ is transcendental over $\mathcal{F}$, we have to have $P_+ , P_-, P_0 \in \mathcal{F}[Q,t]$. 

First we show that both $P_+$ and $P_-$ are at least linear in $Q$. To see this, observe that the $u$ and $v$ coordinates of $(M+D)^k(e_u+e_v)$ contain a single term $Q^k$ and no higher power of $Q$ shows up elsewhere in $(M+D)^k(e_u+e_v)$, nor in $(M+D)^j(e_u+e_v)$ for any $j<k$. Hence if $P_+(t) = t^k + c_{k-1}t^{k-1}+\cdots$ where $c_j \in \mathcal{F}[Q]$, then  we have $ 0 = P_+(M+D)(e_u+e_v) = (M+D)^k (e_u+e_v) + \sum_{j=0}^{k-1} c_j (M+D)^j (e_u+e_v)$. This can only happen if the $Q^k$ coming from the first term is cancelled by something. This can only be if at least some of the $c_j$ coefficients are in $\mathcal{F}[Q] \setminus \mathcal{F}$. Thus $P_+$ needs to be at least linear in $Q$. The same argument shows this for $P_-$ as well.

Note that $\phi_{M+D}$ is quadratic in $Q$, which implies that both $P_+$ and $P_-$ have to be exactly linear in $Q$. Going back to the cancellation of the $Q^k$ term in $(M+D)^k(e_u+e_v)$, we see that since the coefficients $c_j$ are at most linear in $Q$, only the $c_{k-1} (M+D)^{k-1}(e_u+e_v)$ has a chance to cancel the $Q^k$ term, and for this it has to be that $c_{k-1} +Q \in \mathcal{F}$. But $-c_{k-1} = \Tr P_+$, so this implies $\Tr P_+ -Q \in \mathcal{F}$. That the same holds for $P_-$ follows the exact same way.
\end{proof}

We can prove a similar result for diagonal perturbations at a single vertex. This will be useful in some of our constructions in Section~\ref{sec:constructions}.

\begin{lemma}\label{lem:tracew}
Let $M$ be a symmetric matrix with connected support whose elements are in a field $\mathcal{F} \leq \R$. Suppose $u, v \in V(M)$ are cospectral for $M$, and let $w\in V(M)$ be another index such that there is an integer $d \geq 0$ for which $\langle e_w, M^d(e_u+e_v)\rangle \neq 0$. Let $D = Q \cdot D_{w}$ where $Q\in \R$ is transcendental over $\mathcal{F}$, and suppose $u$ and $v$ are also cospectral for $M+D$. Let $\phi_{M+D} = P_+ \cdot P_- \cdot P_0$ as in Lemma~\ref{lem:cospectral_decomposition}. Then $\Tr P_+ - Q \in \mathcal{F}$.
\end{lemma}

\begin{proof}
The argument is almost identical to the previous one. Let $d$ be the smallest power for which $\langle e_w, M^d(e_u+e_v)\rangle \neq 0$. From the setup it follows that $\phi_{M+D} \in \mathcal{F}[Q,t]$ is linear in $Q$, thus the polynomial $P_+ \in \mathcal{F}[Q,t]$ can be at most linear. But it also has to be at least linear, since in $(M+D)^k(e_u+e_v)$ there will be a $Q^{k-d}$ term appearing that would only cancel if one of the $c_j$ coefficients of $P_+$ contains $Q$. This coefficient then can only be $c_{k-1}$ and, as previously, it can only happen if $c_{k-1} + Q  \in \mathcal{F}$, and thus $\Tr P_+ - Q \in \mathcal{F}$.
\end{proof}

\subsection{Achieving irreducibility}

\begin{lemma}\label{lem:irreducible_perturbation}
Let $M$ be a symmetric matrix whose entries are in a field $\mathcal{F} \leq \R$. Assume $u,v \in V(M)$ are \emph{strongly} cospectral indices for $M$. If $Q$ is transcendental over $\mathcal{F}$ and $D =Q \cdot D_{uv},$ then $P_+(M+D)$ and $P_-(M+D)$ are irreducible.
\end{lemma}
\begin{proof}
By Lemma~\ref{lem:strongly_cospectral_perturbation} the indices $u,v$ are also strongly cospectral for $M+D$. Then $\phi_{M+D}$ factors as $\phi_{M+D} = P_+ \cdot P_- \cdot P_0$ by Lemma~\ref{lem:cospectral_decomposition}. Since such a factorization exists for all values of $Q$, it has to be a factorization in $\mathcal{F}[Q,t]$.  
 By expanding the determinant we also have 
\begin{equation}\label{eq:Q_expansion}
\phi_{M+D} = \phi_M - 2Q\phi_{M_u} + Q^2 \phi_{M_{uv}}.
\end{equation}
Thus we see that $\phi_{M+D}$ is quadratic in $Q$, and thus $P_+, P_-$, and $P_0$ can all be at most quadratic in $Q$. Now note that as $Q \to \infty$, there will be two eigenvectors of $M+D$ converging to $\frac{1}{\sqrt{2}}(e_u + e_v)$ and $\frac{1}{\sqrt{2}}(e_u - e_v)$ corresponding to eigenvalues asymptotically equal to $Q$ and all other eigenvalues will be $o(Q)$. This means that $\Tr P_+$ and $\Tr P_-$ will both converge to $Q$. Since these traces as elements of $\mathcal{F}[Q]$ are constant, this can only happen if $\Tr P_+ -Q \in \mathcal{F}$ and $\Tr P_- - Q \in \mathcal{F}$. Thus there are non-zero polynomials $S_\pm, R_\pm \in \mathcal{F}[t]$ such that $P_\pm = S_\pm + Q \cdot R_\pm$, and hence by comparing the degrees in $Q$, we get that $P_0 \in \mathcal{F}[t]$.

Comparing to \eqref{eq:Q_expansion} we see that 
\begin{align*}
\phi_M &= P_0 \cdot S_+ \cdot S_-\\
\phi_{M_u} &= P_0\cdot (S_+R_-+R_+S_-)\\
\phi_{M_{uv}} &= P_0 \cdot R_+ \cdot R_-.
\end{align*}
Now suppose that $P_+$ is not irreducible.  For transcendental $Q$, this implies that $P_+$ factors in $\mathcal{F}[t,Q]$, but since it's linear in $Q$, the only way for this to happen is that there is some factor $T \in \mathcal{F}[t]$ that divides both $S_+$ and $R_+$.  If this is the case, then $T \cdot P_0$ divides all three of $\phi_M,\phi_{M_u}, \phi_{M_{uv}}$.

Let $(t-\lambda)^k$ be a factor of $T \cdot P_0$.  Then by Lemma \ref{lem:vanish}, there are $k$ eigenvectors that vanish on $u$.  By strong cospectrality of $u$ and $v$, these must vanish on $v$ simultaneously, so there are $k$ eigenvectors for $\lambda$ that vanish on both $u$ and $v$. This means that $(t-\lambda)^k$ is already a factor of $P_0$. Since this holds for all factors of $T \cdot P_0$ and hence $T=1$. Thus $P_+$ is irreducible. The same argument gives that $P_-$ is also irreducible.
\end{proof}

\section{Constructions}\label{sec:constructions}
In this section we explain how to obtain graphs with a pair of cospectral nodes $u,v$ where adding a potential $Q$ at nodes $u$ and $v$, and possibly at a third node $w$ results in PGST between $u$ and $v$. The significance of these constructions is that they yield examples without symmetries, in particular without an involution mapping $u$ to $v$.

\subsection{Equitable partitions}\label{sec:eq}

Our first construction is based on equitable partitions. These can be thought of as direct generalizations of graphs with an involution.

\begin{definition}
An equitable partition of a symmetric matrix $M$ is a partition $\PP=\{P_1,\ldots,P_k\}$ of its index set $V(M)$ such that for any $P_i,P_j \in \PP$ and any $v_1,v_2 \in P_i$, one has
\[\sum_{u \in P_j} M(v_1,u) = \sum_{u \in P_j} M(v_2,u).\]
\end{definition}

\begin{theorem}\label{thm:equitable} Let $M$ be a symmetric matrix with connected support whose elements are in a field $\mathcal{F} \leq \R$. Suppose $M$ admits an equitable partition $\PP$ such that $P_1 = \{u,v\}$ and $P_2 = \{w\}$, then for algebraically independent numbers $Q_1,Q_2$ that are transcendental over $\mathcal{F}$ and for $D = Q_1 \cdot D_{uv} + Q_2 \cdot D_w$ the matrix $M+D$ admits PGST between $u$ and $v$. 
\end{theorem}

\begin{corollary}\label{cor:equit}
If a connected graph has an equitable partition with a part consisting of $u,v$ and another part of consisting of $w$ only, then by adding suitable potentials at $u,v$, and $w$ one can guarantee PGST between $u$ and $v$. 
\end{corollary}

\begin{proof}
We proceed step-by-step as follows: first we show that $u$ and $v$ are cospectral in both $M$ and $M+Q_2 \cdot D_w$. Then we show that $u$ and $v$ are strongly cospectral in $M+D$, and furthermore that the corresponding $P_+$ and $P_-$ are irreducible. Finally, we show that $\Tr P_+ / \deg P_+ \neq \Tr P_- /\deg P_-$ hence by Theorem~\ref{thm:tr/deg} there is PGST between $u$ and $v$.

Let us start by proving cospectrality of $u$ and $v$. First, let $\Pi_\PP$ denote the \emph{partition matrix} corresponding to $\PP$. That is, the columns of $\Pi_\PP$ are indexed by $1,2,\dots, k$, and the rows are indexed by $V(M)$, and $\Pi_\PP(j, x)$ is 1 if $x \in P_j$ and 0 otherwise. Second, let $M_\PP$ denote the \emph{quotient} matrix, given as $M_\PP(i,j) = \sum_{y \in P_j} M(x,y)$ for some fixed $x \in P_i$. As $\PP$ is equitable, the value $M_\PP(i,j)$ doesn't depend on the particular choice of $x$. Note, that $M_\PP$ is a $k\times k$ matrix, though not necessarily symmetric. 

A simple computation shows that $M \cdot \Pi_\PP = \Pi_\PP \cdot M_\PP$. Note that, since $P_1 = \{u,v\}$, we can write $e_u+e_v$ as $\Pi_\PP (1,0,\dots,0)^T$. Now we can compute 
\[ \langle e_u - e_v, M^m (e_u+e_v) \rangle = (e_u - e_v)^T M^m \Pi_\PP (1,0,\dots,0)^T = (e_u-e_v)^T \Pi_\PP M_\PP^m (1,0,\dots,0)^T = 0,\]
since $(e_u -e_v)^T \Pi_\PP = 0$. This shows that $W(M,e_u+e_v)$ is orthogonal to $W(M,e_u-e_v)$ and hence $u$ and $v$ are cospectral. 

As $\PP$ is also an equitable partition for $M+Q_2 \cdot D_w$, it follows that $u$ and $v$ are also cospectral for $M+Q_2 \cdot D_w$. Let us write $Q_1 = A+B$ where $A,B, Q_2$ are all algebraically independent of each other and of $\mathcal{F}$. This can be done by choosing $A$ to be independent of $Q_1, Q_2$ and transcendental over $\mathcal{F}$ and then setting $B = Q_1 - A$. Then, by Lemma~\ref{lem:strongly_cospectral_perturbation} and the assumption that $A$ is transcendental over $\mathcal{F}(Q_2)$, we find that $u$ and $v$ are strongly cospectral for $M+Q_2 \cdot D_w+ A \cdot D_{uv}$. Then, by Lemmas~\ref{lem:irreducible_perturbation} and~\ref{lem:strongly_cospectral_perturbation}, $u$ and $v$ are not only strongly cospectral for $M+D = (M+Q_2\cdot D_w + A\cdot D_{uv}) + B\cdot D_{uv}$, but also the corresponding $P_+$ and $P_-$ are irreducible.

Now, by Lemmas~\ref{lem:traceppm} and~\ref{lem:tracew}, we find that $\Tr P_+ - Q_2 \in \mathcal{F}(Q_1)$ and since $\phi_{M+D}$ is linear in $Q_2$ this implies $\Tr P_- \in \mathcal{F}(Q_1)$. Then surely $\Tr P_+/\deg P_+$ cannot equal $\Tr P_- /\deg P_-$ since that would imply $Q_2 \in \mathcal{F}(Q_1)$, a contradiction. 

Finally, by Theorem~\ref{thm:tr/deg} we get that there is PGST between $u$ and $v$.
\end{proof}

\begin{remark}
Given any graph with an equitable partition with a part of size two (and thus a cospectral pair) it is straightforward to add a single vertex and attach it to the vertices of one of the parts of the partition to produce a graph satisfying the conditions of the corollary.  
\end{remark}

\subsection{Gluing}

Our second construction is based on an arbitrary graph $G$ with a pair of cospectral nodes $u,v \in V(G)$. We will show that either simply adding a transcendental potential $Q$ at the nodes $u$ and $v$ induces PGST between them, or else one can modify $G$ in a relatively simple way: by gluing a long path to $G$ with $u$ and $v$ being its endpoints, and then adding a transcendental potential $Q$ at $u$ and $v$ we get PGST between $u$ and $v$. 

\begin{theorem}\label{thm:gluing}
Let $G$ be a graph with $u,v \in V(G)$ cospectral, and such that 0 is not an eigenvalue of the adjacency matrix of $G\setminus \{u,v\}$.  Fix an integer $q \geq 0$. Construct $G_q$ by gluing a path of length $q$ to $G$ by attaching its endpoints to $u$ and $v$. In other words, by adding $q-1$ new nodes $x_1, x_2, \dots, x_{q-1}$ to $G$ together with the edges $u x_1, x_1 x_2, x_2x_3, \dots, x_{q-2}x_{q-1}, x_{q-1} v$.  (For $q=0$ we simply take $G_0 = G$.)

Let $Q \in \R$ be a transcendental number, and put a potential $Q$ at the nodes $u$ and $v$ in $G_q$. Then either the potential induces PGST between $u$ and $v$ in $G$ or there is an infinite set of integers $S \subset \Z$ such that this potential induced PGST between $u$ and $v$ in $G_q$ for all $q \in S$.  
\end{theorem}

Again, the main novelty of this construction is that it does not require the graph to admit any kind of symmetry. In fact, one can start from any graph with a pair of cospectral nodes, of which many examples have been described in the literature. We give the proof at the end of this section.

We begin by describing a general gluing construction that preserves cospectrality. This has been independently discovered by Godsil~\cite{godsil_newbook}. As we have done so far, we will prove everything in the general context of symmetric matrices, but we are still primarily interested in the case where the matrices in question are the adjacency matrices of graphs.

Let $M_1$ and $M_2$ be symmetric matrices such that $V(M_1) \cap V(M_2) = \{u,v\}$. We can extend them to matrices $\M_1, \M_2$ on $V(M_1)\cup V(M_2)$ by declaring them to be 0 wherever they weren't previously defined. Then we define their sum $M_1 \oplus M_2 = \M_1 + \M_2$, in particular $V(M_1 \oplus M_2) = V(M_1) \cup V(M_2)$.
When $M_i$ is the adjacency matrix of the graph $G_i$ ($i=1,2$), then $M = M_1 \oplus M_2$ is the adjacency matrix of $G = G_1 \cup_{uv} G_2$ sometimes referred to as the 2-sum of $G_1$ and $G_2$, that is obtained by gluing the two $u$ nodes together and the two $v$ nodes together. Note that $G$ may have multiple edges. 

\begin{lemma}\label{lem:glue}
If $u, v \in V(M_i); (i=1,2)$  are cospectral pairs for both $M_1$ and $M_2$, then they are also cospectral in $M_1 \oplus M_2$. 
\end{lemma}

\begin{proof}
Let $M = M_1 \oplus M_2$.
We can compute $\phi_{M_u}$ by expanding the determinant along the column corresponding to $v$:
\[ \phi_{M_u} = \phi_{{M_1}_u} \phi_{{M_2} _{uv}} + \phi_{{M_2}_u} \phi_{{M_1}_{uv}} - t \phi_{{M_1}_{uv}} \phi_{{M_2}_{uv}}.\] 
By cospectrality $\phi_{{M_j}_u} = \phi_{{M_j}_v}$, and thus the right hand side doesn't change when exchanging the roles of $u$ and $v$. Hence $\phi_{M_u} = \phi_{M_v}$ as claimed.
\end{proof}

In what follows we assume that $u$ and $v$ are indeed cospectral in $M_1$ and in $M_2$ and let $M = M_1 \oplus M_2$. Let us introduce the notation $\phi_{M_j} = P_+^j \cdot P_-^j \cdot P_0^j  \; (j=1,2)$ and $\phi_M = P_+ \cdot P_- \cdot P_0$, according to Lemma~\ref{lem:cospectral_decomposition}.

\begin{lemma}\label{lem:p+-_deg_subadditive}
$\deg P_+ \leq \deg P_+^1 + \deg P_+^2 - 1$ and $\deg P_- \leq \deg P_-^1 + \deg P_-^2 - 1$
\end{lemma}

\begin{proof}
We know, by Definitions~\ref{def:ppm} and~\ref{def:wmz} that $\deg P_+ = \dim W(M, e_u+e_v)$ and $\deg P_+^j = \dim W(\M_j, e_u+e_v)$. We will show that $W(M,e_u+e_v) \leq W(\M_1,e_u+e_v) \oplus W(\M_{G_2},e_u+e_v)$. From this, the first part of the lemma will follow since $\langle e_u+e_v \rangle \leq W(\M_1, e_u+e_v) \cap W(\M_2,e_u+e_v)$.

 Let us denote by $\Pi_1, \Pi_2$, and $\Pi_0$ the ``natural'' projection operators from $\R^G$ to $\R^{G_1}, \R^{G_2}$, and $\R^{\{u,v\}}$ respectively.
First note that, by cospectrality, $e_u M^k (e_u+e_v) = e_v M^k (e_u+e_v)$ for any $k$. In other words, $e_u+e_v$ is an eigenvector of $\Pi_0 M^k$ for any $k$. The same is true with $\M_1$ or $\M_2$ in place of $M$. Also note that a simple computation gives $\M_1 \M_2 = \M_1 \Pi_0 \M_2$ and $\M_2 \M_1 = \M_2 \Pi_0 \M_1$.

It is then sufficient to prove that $\Pi_1 M^k (e_u+e_v) \in W(\M_1, e_u+e_v)$ and $\Pi_2 M^k (e_u+e_v) \in W(\M_2, e_u+e_v)$. Without loss of generality it is sufficient to prove the first one. Using $M = \M_1+\M_2$ we can compute 
\[ M^k = \sum_{j=0}^k \left(\M_1^j \sum_{\stackrel{0<j_1,j_2,\dots}{j+j_1+j_2+\dots = k}}\M_2^{j_1} \M_1^{j_2}\M_2^{j_3}\M_1^{j_4}\dots\right)\] 
and so 
\[ \Pi_1 M^k (e_u+e_v) = \sum_{j=0}^k \left(\Pi_1 \M_1^j \sum_{\stackrel{0<j_1,j_2,\dots}{j+j_1+j_2+\dots = k}}\Pi_0 \M_2^{j_1} \Pi_0 \M_1^{j_2} \Pi_0 \M_2^{j_3} \Pi_0 \M_1^{j_4}\dots (e_u+e_v)\right).\] 
Here, each term in the sum is just a multiple of $(e_u+e_v)$ since it is an eigenvector of each $\Pi_0 \M_\epsilon^j : \epsilon = 1,2$. Hence there are constants $c_j$ depending only on $j$ and $k$ such that 
\[ \Pi_1 M^k (e_u+e_v) = \sum_{j=0}^k c_j \Pi_1 \M_1^j (e_u+e_v) = \sum_{j=0}^k c_j \left(\M_1^j(e_u+e_v) - \Pi_0 \M_1^j(e_u+e_v)\right) \in W(\M_1,e_u+e_v),\] and this is what we wanted to show.

The argument for $P_-$ is analogous. 
\end{proof}

\begin{remark}
Any eigenvector of $M_1$ or $M_2$ that vanishes on $u,v$ can be extended to $V(M)$ by zeros to obtain an eigenvector of $M$ with the same eigenvalue. Thus $P_0$ is divisible by $P_0^1 \cdot P_0^2$.
\end{remark}

\begin{lemma}
Let $k_j \geq 0$ denote the multiplicity of $\lambda$ in $P_0^j \; (j=1,2)$. Suppose the multiplicity of $\lambda$ in $P_0$ is strictly bigger than $k_1+k_2$. Then $\lambda$ is an eigenvalue of ${M_1}_{uv}$ and ${M_2}_{uv}$. 
\end{lemma}

\begin{proof}
By the assumption on the multiplicity there has to be an eigenvector of $M$ with eigenvalue $\lambda$ vanishing on both $u$ and $v$ that is not identically zero on either $M_1$ or $M_2$. The restriction of this vector to $V(M_1)\setminus \{u,v\}$ and to $V(M_2) \setminus \{u,v\}$ then yield eigenvectors showing that $\lambda$ is indeed an eigenvalue of both of these matrices.
\end{proof}

\begin{corollary}\label{cor:p0_deg_additive}
If ${M_1}_{uv}$ and ${M_2}_{uv}$ do not share any eigenvalues, then $P_0 = P_0^1 \cdot P_0^2$.
\end{corollary}

\begin{theorem}\label{thm:degppm}
If ${M_1}_{uv}$ and ${M_2}_{uv}$ do not share any eigenvalues, then $\deg P_+ = \deg P_+^1 + \deg P_+^2 - 1$ and $\deg P_- = \deg P_-^1 + \deg P_-^2 -1$.
\end{theorem}

\begin{proof}
By Corollary~\ref{cor:p0_deg_additive} and by Lemma~\ref{lem:p+-_deg_subadditive} we have 
\begin{multline*} |V(M)| = \deg P_0 + \deg P_+ + \deg P_- \leq \\ \leq  \deg P^1_0 + \deg P^2_0 + \deg P_+^1 + \deg P_+^2 - 1 + \deg P_-^1 + \deg P_-^2 -1  = |V(M_1)| + |V(M_2)| - 2 = |V(M)|\end{multline*}
Since the left and right hand sides are equal, there must be equality in the middle, finishing the proof.
\end{proof}

\begin{proof}[Proof of Theorem~\ref{thm:gluing}]
Let $A$ denote the adjacency matrix of $G$. By assumption $u$ and $v$ are cospectral for $A$. The matrix $H_G = A + Q \cdot D_{uv}$ is the Hamiltonian for the graph $G$ together with the potential. By Lemmas~\ref{lem:strongly_cospectral_perturbation} and~\ref{lem:irreducible_perturbation} we know that $u$ and $v$ are strongly cospectral for $H_G$ and the corresponding $P^{H_G}_+$ and $P^{H_G}_-$ polynomials are irreducible, and by Lemma~\ref{lem:traceppm} we know that $\Tr P^{H_G}_+ - Q$ and $\Tr P^{H_G}_- -Q$ are both rational. (To show irreducibility we need to apply the same trick as in the proof of Theorem~\ref{thm:equitable}: adding the potential in two steps, first ensuring strong cospectrality, then irreducibility.) So by Theorem~\ref{thm:tr/deg}, the only way there could be no PGST between $u$ and $v$ is if \[\deg P^{H_G}_+ = \deg P^{H_G}_-.\] 

Let now $\PPP_q$ denote the path graph on $q+1$ nodes, and let $A_q$ denote its adjacency matrix. Let us call the endpoints $u$ and $v$. It is clear that $u$ and $v$ are cospectral in $\PPP_q$, for instance because $\PPP_q$ admits an equitable partition, each part consisting of a pair of symmetric nodes, or the single node in the middle.

Then if $G_q = G \cup_{u,v} \PPP_q$ then $M= A \oplus A_q$ is the adjacency matrix of $G_q$. Finally let $H = M + Q \cdot D_{uv}$ denote the Hamiltonian of $G_q$ together with the potential. Then $H = M_1 \oplus M_2$ where $M_1 = A + Q \cdot D_{uv}$ and $M_2 = A_q$. It is well-known that the eigenvalues of ${A_q}_{uv} = A_{q-1}$ are $2\cos(j \pi/q) \; (j=1,2,\dots, q-1)$. It is also not hard to show that $\deg P_+^2 = \lceil (q+1)/2 \rceil$ and $\deg P_-^2 = \lfloor (q+1)/2\rfloor$. 

Any non-zero real number $\lambda$ there is at most one prime $p$ such that $\lambda = 2\cos( j \pi/ (2p))$ for some $1\leq j \leq 2p-1$, and since 0 is not an eigenvalue of $A_{uv}$ by assumption. Thus if $p$ is a sufficiently large prime number and $q=2p$, then ${A_q}_{uv}$ and $A_{uv}$ do not share any eigenvalues. Then, by Theorem~\ref{thm:degppm} we find that \[\deg P^H_+ = \deg P^{M_1}_+ + \deg P^{M_2}_+ -1 = \deg P^{H_G}_+ + \lceil (2p+1)/2 \rceil - 1 = \deg P^{H_G}_+ + p\] and \[\deg P^H_- = \deg P^{M_1}_- + \deg P^{M_2}_- -1 = \deg P^{H_G}_- + \lfloor (2p+1)/2 \rfloor - 1 = \deg P^{H_G}_- + p - 1,\] so $\deg P^H_+ \neq \deg P^H_-$. At the same time $\Tr P^H_+ - Q$ and $\Tr P^H_- - Q$ are both rational, and $u$ and $v$ are strongly cospectral and $P^H_+$ and $P^H_-$ are irreducible, as before. So by Theorem~\ref{thm:tr/deg} there is PGST between $u$ and $v$ in $G_q$.
\end{proof}

We can in fact remove the condition of Theorem \ref{thm:gluing} that 0 not be an eigenvalue of $A_{uv}$ if we allow potential to be placed on vertices other than $u$ and $v$ (the two cospectral vertices).  This is the content of the next two theorems.

\begin{theorem}\label{thm:glue_pot}
Let $G$ be a graph with $u,v\in V(G)$ cospectral.  Let $k$ be any odd integer and let $\PPP_k$ denote the path on $q$ nodes, and call its endpoints $u,v$.  Add a suitably chosen potential to every vertex of $\PPP_k$ so that $G\diff{u,v}$ shares no eigenvalues with $\PPP_k\diff{u,v}$.  Create $G'$ by gluing the path with potential to the nodes $u$ and $v$. Then putting a transcendental potential $Q$ on $u$ and $v$ induces PGST from $u$ to $v$ in $G'$.
\end{theorem}
\begin{proof}
Adding a potential to every vertex of $\PPP_k$ simply adds a multiple of the identity to its adjacency matrix, so the eigenvalues shift by the amount of the potential.  Thus clearly a potential can be chosen so that $G\diff\{u,v\}$ and $\PPP_k\diff\{u,v\}$ do not share any eigenvalues.  Then the proof proceeds exactly as in the proof of Theorem \ref{thm:gluing} to show that there is PGST.
\end{proof}

\begin{theorem}\label{thm:change_trace}
Let $G$ be a graph with $u,v\in V(G)$ cospectral.  Let $k$ be any odd integer and let $\PPP_k$ denote the path on $q$ nodes, and call its endpoints $u,v$.  Denotes its central vertex by $w$.  Add a transcendental potential $Q'$ to $w$ and then create $G'$ by gluing the path with this potential to the nodes $u$ and $v$. Then putting a transcendental potential $Q$ algebraically independent from $Q'$ on $u$ and $v$ induces PGST from $u$ to $v$ in $G'$.
\end{theorem}
\begin{proof}
By Lemma \ref{lem:tracew}, $Q'$ appears in $Tr(P_+)$ but not in $Tr(P_-)$, but $Q'$ is algebraically independent from any other terms that could show up in the trace, so it must be that $Tr(P_+)$ and $Tr(P_-)$ are distinct.  Then the theorem follows from Theorem \ref{thm:tr/deg}.
\end{proof}

\section{Examples, discussion, and further questions}

Our results succeed in giving infinite families of graphs for which we can put a potential on the vertices to induce PGST between two vertices.  Furthermore, the potential required can be assumed to be zero on most vertices of the graph.  In addition, the examples produced do not require any of the strict symmetry or regularity conditions of the results in \cite{invol} and \cite{srg}.  We will examine the some examples, including the graphs shown in the introduction.

\begin{example}
Let $G$ be the graph below.
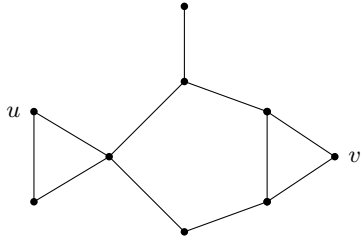
\begin{center}
\begin{tikzpicture}
\draw (0,0)node{}--(-1,.6)node{}--(-1,-.6)node{}--(0,0)--(1,1)node{}--(2.1,.6)node{}--(2.1,-.6)node{}--(1,-1)node{}--(0,0)
(1,1)--(1,2)node{}
(2.1,.6)node{}--(3,0)node{}--(2.1,-.6)node{}
(-1.1,.6)node[fill=white, left]{\small $u$}
(3.1,0)node[fill=white, right]{\small $v$};
\end{tikzpicture}
\end{center}
Direct computation can show that vertices $u$ and $v$ are cospectral in $G$ (but \emph{not} strongly cospectral).  Putting a transcendental potential $Q$ on $u$ and $v$ makes these vertices strongly cospectral by Lemma \ref{lem:strongly_cospectral_perturbation}, and in fact $P_+$ and $P_-$ have different degree in this case, so this potential is enough to obtain pretty good state transfer.  Gluing paths with an even number of vertices gives an infinite family of graphs for which the potential induces pretty good state, and each graph in this family does not have an automorphism mapping $u$ to $v$.  

Note that we chose paths of even length simply because we know that these change the degree of $P_+$ and $P_-$ by the same amount, and this graph already has $deg(P_+)\neq deg(P_-)$.  We could in fact glue \emph{any} graph with a pair of cospectral vertices as long as the resulting graph has $P_+$ and $P_-$ with distinct degree or trace, and achieve a graph for which the potential induces pretty good state transfer.
\end{example}

\begin{example}
Let $G$ be the graph shown below.
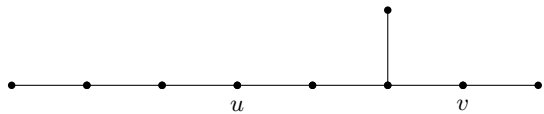
\begin{center}
\begin{tikzpicture}
\draw \foreach \x in {0,1,2,3,4,5,6}
{
(\x,0)node{}--(\x+1,0)node{}
}
(5,0)--(5,1)node{}
(3,-.1)node[fill=white, below]{\small $u$}
(6,-.1)node[fill=white, below]{\small $v$};
\end{tikzpicture}
\end{center}
Here, by direct computation, we have $deg(P_+)=deg(P_-)$ and $Tr(P_+)=Tr(P_-)=0$, so in order for our results to give PGST, we need to use Theorem \ref{thm:glue_pot} or \ref{thm:change_trace}. 
\end{example}

We pose the natural question: given any pair of cospectral vertices $u$ and $v$, can we always induce PGST by a potential placed only on vertices $u$ and $v$?  We can answer this question in the negative with the following example.

\begin{example}
Consider the graph pictured below, with the vertices $u,v$ as labeled. 

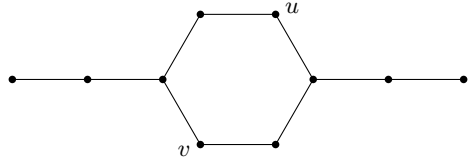
\begin{center}
\begin{tikzpicture}
\draw (0:1)node{}--(60:1)node{}--(120:1)node{}--(180:1)node{}--(240:1)node{}--(300:1)node{}--(0:1)
(0:1)--(0:2)node{}--(0:3)node{}
(180:1)--(180:2)node{}--(180:3)node{}
(60:1.1)node[fill=white, right]{\small $u$}
(240:1.1)node[fill=white, left]{\small $v$};
\end{tikzpicture}
\end{center}
Computation shows that $deg(P_+)=deg(P_-) = 5$ and $Tr(P_+)=Tr(P_-)=Q$, where $Q$ is the value of the potential on $u$ and $v$.  So Theorem \ref{thm:tr/deg} is uninformative.  But using Lemma \ref{lem:eig} directly, since the degrees of $P_+$ and $P_-$ are odd, we can simply take $\ell_i=1$ for each $i$ and $m_j=-1$ for each $j$, and we will have an integer linear combination of the eigenvalues equal to 0 with $\sum \ell_i$ and $\sum m_j$ odd. Thus, no matter what value of potential we put at $u$ and $v$, there cannot be PGST between $u$ and $v$.

The question remains open if we can induce PGST by putting potential on other vertices as well, since this could in theory change the degrees of $P_+$ and $P_-$.  

To create an infinite family of graphs in which PGST occurs, we can glue paths to this graph via Theorem \ref{thm:gluing}.

Note that this graph has an involution swapping $u$ and $v$ that fixes no vertices or edges (see \cite{invol}) and with an odd number of orbits.  This is the only situation we are aware of where there is a cospectral pair, and we can prove that no potential on $u$ and $v$ can induce PGST.  It is an open question if this is the only kind of such graphs.
\end{example}

A further question concerns the algebraic complexity of the potential necessary to induce PGST.  In all of our results, we have used transcendental values of potential.  This accomplishes two things: first, we can turn any pair of cospectral vertices into a strongly cospectral pair (Lemma \ref{lem:strongly_cospectral_perturbation}), and further, this guarantees that $P_+$ and $P_-$ are irreducible polynomials (Lemma \ref{lem:irreducible_perturbation}; note that irreduciibility is necessary to apply Theorem \ref{thm:tr/deg}).  However, the assumption of a transcendental potential is a drawback in terms of practical considerations.  It is of interest to determine if simpler (algebraic, ideally rational) potentials might do as well.  

\section{Appendix}

Here we prove some lemmas used in the paper.

\begin{lemma}\label{lem:field_trace}
The field trace map $\Tr_{K/F}:K\rightarrow F$ defined in the proof of Theorem \ref{thm:tr/deg} satisfies the following properties:
\begin{itemize}
\item $\Tr_{K/F}$ is an $F$-linear map.
\item For $\alpha\in F$, $\Tr_{K/F}(\alpha) = [K:F]\alpha$.
\item For $K$ and extension of $L$, and extension of $F$, we have $\Tr_{K/F} = \Tr_{L/F}\circ\Tr_{K/L}$.
\end{itemize}
\end{lemma}
\begin{proof}
For a field extension $K$ of $F$, recall the definition of the field trace is, for $\alpha\in K$ is
\[
\Tr_{K/F}(\alpha) = \sum_{g\in Gal(K/F)}g(\alpha).
\]

The linearity over $F$ is clear from the definition.

The second property follows since any automorphism in $K/F$ fixes any element of $F$.

Finally, the last follows from the definition and the Galois correspondence between subfields of $K$ fixing $F$ and subgroups of $Gal(K/F)$.
\end{proof}

\begin{lemma}\label{lem:vanish}
Let $M$ be any real symmetric $n\times n$ matrix, and let $u$ be an index for $M$.
Suppose $\lambda$ is an eigenvalue of multiplicity at least $k$ of both $M$ and $M_u$.  Then there are $k$ linearly independent eigenvectors of $M$ corresponding to $\lambda$ that vanish at $u$.  
\end{lemma}
\begin{proof}
If $\lambda$ has multiplicity strictly larger than $k$ as an eigenvalue of $M$, then it is easy to see that we can adjust a basis for the eigenspace so that at least $k$ of the corresponding eigenvectors vanish at $u$. 

So let us suppose that the multiplicity of $\lambda$ as an eigenvalue of $M$ is exactly $k$, and as an eigenvalue of $M$ is at least $k$. Let us denote by $\lambda_1\leq\cdots\leq\lambda_n$  the eigenvalues of $M$, and by $\mu_1\leq\cdots\leq\mu_{n-1}$ the eigenvalues of $M_u$.
Then the interlacing theorem for symmetric matrices  (see for example Theorem 4.3.8 of \cite{matrix_analysis}) says we have
\[
\lambda_1\leq\mu_1\leq\lambda_2\leq\cdots\leq\lambda_{n-1}\leq\mu_{n-1}\leq\lambda_n.
\]
Then, given the assumption on the multiplicity of $\lambda$ above, we have $\lambda_{j-1} < \lambda = \lambda_j = \lambda_{j+1} = \dots \lambda_{j+k-1}  <  \lambda_{j+k}$. There are two possibilities for the $\mu$-s:
\[ \lambda  = \mu_j = \dots = \mu_{j+k-1} \mbox{\hskip 1cm  or  \hskip 1cm}
  \lambda = \mu_{j-1} = \dots = \mu_{j+k-2} \] 
We will consider the first possibility, the second one can be dealt with in a similar fashion. Let us choose an orthonormal basis $(\varphi_k)_{k=1}^n$ of eigenvectors of $M$ in such a way that $\varphi_j, \varphi_{j+1}, \dots, \varphi_{j+k-2}$ all vanish on $u$. This can be done since the multiplicity of $\lambda$ is $k$ and we are only asking for the first $k-1$ corresponding eigenvectors to vanish on $u$. Then, by the min-max principle, we have

\begin{align*}
\lambda_{j+k-1} &= \min_{\substack{x\neq0,x\in\R^n\\ x\perp \varphi_1,...,\varphi_{j+k-2}}}\frac{x^TMx}{x^Tx}\\
&\leq \min_{\substack{x\neq0,x\in\R^n\\ x\perp \varphi_1,...,\varphi_{j+k-2}\\x(u)=0}}\frac{x^TMx}{x^Tx}= \min_{\substack{x\neq0,x\in\R^{n-1}\\ x\perp \tilde{\varphi}_1,...,\tilde{\varphi}_{j+k-2}}}\frac{x^TM_u x}{x^Tx} \\
&\leq \max_{y_1,...,y_{j+k-2}\in\R^{n-1}}\min_{\substack{x\neq0,x\in\R^{n-1}\\ x\perp y_1,...,y_{j+k-2}\\x(u)=0}}\frac{x^TM_ux}{x^Tx} = \mu_{j+k-1} = \lambda_{j+k-1}.
\end{align*}
This implies that the first inequality has to be equality, so there is an $x$ attaining the minimum that is orthogonal to $\varphi_1, \dots, \varphi_{j+k-2}$ and for which $x(u)=0$. This $x$ then has to be an eigenvector with eigenvalue $\lambda_{j+k-1} = \lambda$, so we exhibited $k$ pairwise orthogonal eigenvectors for $\lambda$ vanishing on $u$.

The case when $\lambda = \mu_{j-1} = \dots = \mu_{j+k-2}$ is done similarly, except we use the characterization of $\lambda_j$ as a maximum, and we fix $\varphi_{j+1}, \dots, \varphi_{j+k-1}$ to vanish on $u$.
\end{proof}



\end{document}